\documentclass[10pt,a4paper,twoside]{amsart}

\usepackage{amsmath}
\usepackage{amsfonts}
\usepackage{amssymb}
\usepackage[all]{xy}
\usepackage{color}

\newtheorem{theorem}{Theorem}[section]
\newtheorem{lemma}{Lemma}[section]
\newtheorem{proposition}{Proposition}[section]
\newtheorem{corollary}{Corollary}[section]
\newtheorem{defn}{Definition}[section]
\newtheorem{remark}{Remark}[section]
\newtheorem{claim}{Claim}[section]

\newtheorem{example}{Example}[section]
\author{Daisuke Matsushita}
\title{On almost holomorphic Lagrangian fibrations}

\address{Division of Mathematics, Graduate School of Science,
         Hokkaido University,  Sapporo, 060-0810 Japan}
\thanks{* Partially supported by Grand-in-Aid \# 18684001
 (Japan Society for Promortion of Sciences).} 
\email{matusita@math.sci.hokudai.ac.jp}

\begin{document}

\begin{abstract}
We prove that the pull back of an ample line bundle
by an almost holomorphic Lagrangian fibration is nef.
As an application, we show birational semi rigidity of Lagrangian
fibrations.
\end{abstract}
\maketitle
\section{Introduction}

We start with the definition of the almost holomorphic
fibration.
\begin{defn}
A rational map $f$ is said to be 
almost holomorphic if 
there exists an open set $U$
on which $f$ is defined and
the induced morphism $U \to S$
is proper.
\end{defn}
We recall the definition of the pull back of a Cartier 
divisor by a rational map.
\begin{defn}\label{pull_back}
Let $f:X\dasharrow S$ be a rational map and $D_S$ a
Cartier divisor on $S$. We define the pull back $f^*D_S$ by
$\nu_*g^*D_S$, where $\nu : Y\to X$ is a resolution of the
indeterminacy of $f$ and $g$ is the induced morphism 
$Y\to S$.  Note that this definition does not depend on the choice
of resolution because every push-forwards of exceptional divisors are zero.
We could consult \cite[\S 2.e]{MR2104208} the fundamental facts of
the push-forwards of divisors.
\end{defn}

\begin{theorem}\label{main}
Let $f:X\dasharrow S$ be an almost holomorphic fibration
from a projective symplectic manifold. Assume that
a general fibre of $f$ is a Lagrangian submanifold.
Let $D_S$ be an ample divisor of $S$. Then the pull back
$f^*D_S$ is nef.
\end{theorem}

\begin{remark}
In the case that $X$ is a non projective
irreducible symplectic manifold, the same assertion holds 
by \cite[Theorem 1.2]{MR2739808}.
\end{remark}

\begin{remark}
 Professors Daniel Greb, Christian Lehn and S\"onke Rollenske
   pointed out that Theorem \ref{main}
 holds without the assumption that a general fibre is Lagrangian
 if $X$ is an irreducible symplectic manifold. Indeed, 
 by Proposition \ref{movable}, there exists a projective irreducible 
 symplectic manifold $X_1$ and a birational map $\phi : X \dasharrow X_1$
 such that $\phi_*(f^*D_S)$ is nef. Since $\phi$ is isomorphic  
 in a neighbourhood of a general fibre of $f : X\dasharrow S$, $X_1$ admits
 also an almost holomorphic fibration and $\phi_*(f^*D_S)$ is trivial on a general
 fibre. Thus the nef dimension of $\phi_*(f^*D_S)$ is strictly less than $\dim X_1$.
 By \cite[Theorem 1.1]{nefdimension}, 
 $\phi_*(f^*D_S)$ is semiample and induces a Lagrangian fibration $f_1 : X_1 \to S_1$.
 A general fibre of $f$ is the complete intersection of 
 general members of the linear system of $|f^*D_S|$
 and a general fibre of $f_1$ is also the complete intersection of
 general members of the linear system of $|\phi_*(f^*D_S)|$.
 Thus a general fibres of $f$   
 and $f_1$ coincide. This implies that a general fibre of $f$ is Lagrangian.
\end{remark}

The following is a geometric interpretation of 
Theorem \ref{main}.
\begin{corollary}\label{geometric_interpretation}
Under the same assumption of Theorem
\ref{main}, we let $U$ be
the maximal open set of $X$ on which $f$ is defined.
Then there exists a Lagrangian fibration $g:X\to T$
which satisfies the following diagram:
\[\xymatrix{
X\ar[d]_g &U\ar@{_{(}->}[l]\ar[d]^f \\
T &\ar@{_{(}->}[l]f(U)
}\]
\end{corollary}

\begin{remark}
There exists an example of an almost holomorphic 
Lagrangian fibration which is not holomorphic.
Let $f:X\to \mathbb{P}^2$ be a Lagrangian fibration.
The composition of $f$ and
a Cremona transformation of 
$\mathbb{P}^2 \dasharrow \mathbb{P}^2$ gives 
such an example. 
\end{remark}

\begin{remark}
 Let $X$ be an irreducible symplectic manifold which
 contains a Lagrangian submanifold $A$ isomorphic to a complex torus.
 In \cite[Question 1.6]{1002.4321},
 it is asked whether $X$ admits a Lagrangian fibration under this condition.
 In \cite[Theorem 4.1]{1105.3410}, they proved that $X$ admits a Lagrangian fibration
 if $X$ is not projective. In the case that $X$ is projective, it is proved that
 $X$ admits an almost holomorphic fibration whose general fibre is a Lagrangian 
 submanifold by \cite[Theorem 3]{1110.2852} if $\dim X = 4$ and
 \cite[Theorem 1.2]{1201.2369} in any dimension. 
 In the case $\dim X = 4$, they proved that $X$ admits a Lagrangian
 fibration in \cite[Theorem 5.1]{1110.2680}.
 Together with Corollary 
 \ref{geometric_interpretation}, $X$ admits a Lagrangian fibration in all cases.
\end{remark}
 
 As an application, we obtain semi birational rigirity of Lagrangian
 fibrations.
 
\begin{corollary}\label{semi-rigidity}
Let $f_1 : X_1 \to S_1$ and $f_2 : X_2 \to S_2$
be Lagrangian fibrations from
projective symplectic
manifolds. Assume that they satisfy the following
diagram
\[
\xymatrix{
X_1 \ar[d] & X_2 \ar@{-->}[l]_{\phi_1} \ar[d] \\
S_1 &  S_2 \ar@{-->}[l]^{\phi_2}
}
\]
where $\phi_1$ and $\phi_2$ are birational maps.
Then $\phi_2$ is isomorphic.
\end{corollary}

\begin{example}\label{Oguiso}
 Let $A$ be an abelian surface with an elliptic fibration 
 $\varphi : A \to F$. The generalized Kummer variety $\mathrm{Km}^{n+1}A$ admits
 a Lagrangian fibration $\mathrm{Km}^{n+1}A \to \mathbb{P}^{n}$.
 Let us consider the intersection of $\mathrm{Km}^{n+1}A$ and
 the relative Hilbert scheme $\mathrm{Hilb}^{n+1}(A/F)$
 in $\mathrm{Hilb}^{n+1}A$.
 We denote by $f$ a torsion point of $F$ of order $n+1$ and by $E$ the
 fibre of $\varphi$ at $f$. We define the map 
 \[
 h_f : \mathrm{Hilb}^{n+1}E \to \mathrm{Ker}(\varphi),
 \]
 by sending a zero cycle of $E$ into its sum of $A$.
 By \cite[Proposition 6.4]{1010.0820}, $P:= \mathrm{Hilb}^{n+1}(A/F)\cap \mathrm{Km}^{n+1}A$
 consists of the disjoint union of the preimage $h_f^{-1}(0)$. Hence
 $P$ is the disjoint union of $2^{n+1}$ copies $\mathbb{P}^{n}$ and
 is contained in fibres of the Lagrangian fibration. The Mukai flop 
 $(\mathrm{Km}^{n+1}A)'$
 along $P$ admits a Lagrangian fibration and we have the following diagram:
 \[
 \xymatrix{
 \mathrm{Km}^{n+1}A  \ar[d] & (\mathrm{Km}^{n+1}A)' \ar@{<--}[l]_{\phi} \ar[d] \\
 \mathbb{P}^{n} & \ar@{=}[l] \mathbb{P}^{n}
 }
 \]  
 where $\phi$ is not isomorphic.
\end{example}

\section*{Acknowledgement} 
The author would like to express
his thanks to Professors Christian Lehn 
and Ekaterina Amerik 
for their comments.
He also would like to express his thanks to 
Professor Oguiso who taught him Example \ref{Oguiso}.

\section{Proof of Theorem \ref{main}}

\begin{proof}[Proof of Theorem \ref{main}]
We start with proving the following Proposition.
\begin{proposition}\label{movable} 
 Let $X$ be a projective symplectic manifold and $D$ an effective divisor on $X$.
 Assume that the linear system $|D|$ associated with $D$ has no fixed divisor. Then
 there exists a birational map $\phi : X \dasharrow X'$ and a projective symplectic manifold
 $X'$ such that $\phi_*D$ is nef and $\phi$ is isomorphic on the open set
 $X\setminus \mathrm{Bs}|D|$, where
 $\mathrm{Bs}|D|$ stands for the base locus of the linear system $|D|$.
\end{proposition}

\begin{proof}
 If $D$ is nef, there are nothing to proof. 
 Assume that $D$ is not nef. We choose a small rational
 number $\delta$ such that the pair
 $(X,\delta D)$ has only Kawamata log terminal singularities.
By \cite[Theorem 3-2-1]{MR946243}, 
there exists an extremal ray of the
pair $(X,\delta D)$
and we have a contraction morphism $\pi^-_0 : X \to Z_0$
such that $-D$ is $\pi^-_0$-ample
by \cite[Theorem 4-2-1]{MR946243}.
If we choose an another effective ample divisor $D'$
which is linearly equivalent to $D$, $-D'$ is also $\pi^-_0$-ample.
Hence
the exceptional locus of $\pi^-_0$
is contained in the base locus of $|D|$. Since $|D|$
has no fixed component, $\pi^-_0$ is small.
By \cite[Corollary 1.4.1]{MR2601039}, 
we have a birational morphism 
$\pi^{+}_0 : X_1 \to Z_0$ such that
$(\phi_1)_*D$ is $\pi^+_0$-ample, where $\phi_1$ is
the induced
birational map $X \dasharrow X_1$. 
If $(\phi_1)_*D$
is not nef, we repeat this process
and obtain the following diagram:
\[
 \xymatrix{
  X \ar[rd]_{\pi^-_0} & &X_1 \ar[ld]^{\pi^+_0} 
  \ar@{<--}[ll]_{\phi_1} \ar[rd]_{\pi^-_1} && X_2 
  \ar[ld]^{\pi^+_1} 
  \ar@{<--}[ll]_{\phi_2} 
  \ar[rd]_{\pi^-_2} &&X_3 \ar[ld]^{\pi^+_2}
     \ar@{<--}[ll]_{\phi_3} &\cdots&  X_n 
     \ar[ld]^{\pi^+_{n-1}} 
       \\
  & Z_0  & & Z_1 & & 
  Z_2 & \cdots& Z_{n-1}  
 }
\]
 We denote $(\phi_{i})_*D_{i-1}$ by
 $D_i$, where $D_0 := D$.  We also denote
 by $\phi_{i} : X_{i-1} \dasharrow X_{i}$ the induced
 birational maps.
 Since $\pi_{i}^{-}$ is $-D_{i}$ ample, the exceptional locus
 of $\pi_{i}^{-}$ is contained in the base locus of $|D_{i}|$ and $\phi_{i}$
 is isomorphic over $X_{i-1}\setminus \mathrm{Bs}|D_{i-1}|$.
 Let $U_i$ be the image 
 $X_{i-1}\setminus \mathrm{Bs}|D_{i-1}|$ by $\phi_i$. 
 Since 
 $\mathrm{Bs}|D_i| \cap U_i = \emptyset$, $\phi_{i+1}$
 is defined and isomorphic over $U_{i}$. Hence 
 the composition maps $\phi_{n}\circ \cdots \circ \phi_{1}$ are defined and
 isomorphic over $U_{0}=X\setminus \mathrm{Bs}|D|$ for all $n$.
 The following Lemma completes the proof of Proposition.
\begin{lemma}
\begin{itemize}
\item[(1)] All $X_i$ are projective symplectic manifolds
\item[(2)] After finite steps, $D_n$ becomes nef.
\end{itemize}
\end{lemma} 
\begin{proof}
 The assertion (1) follows by 
 \cite[Corollary 1]{namikawa_q_factorial_deformation}.
 Assume that $D_n$ is not nef for all $n$. Then we have
 an infinite chain of $D$-flops. 
 By \cite[Theorem 1.2]{appendix},
 there exist no such chains. For reader's convenience,
 we give a brief explanation.
 According to \cite[Theorem]{Shokurov},
 such a chain does not exist if the following two conditions are
 satisfied:
 \begin{itemize}
 \item[(1)] The function $p_i(x)$ on $X_i$ defined by
 \[
 X_i \ni x \mapsto p_i (x) := \mathrm{mld}(x;X_i,\delta D_i)
 \]
 is lower semicontinuous.
 
 \item[(2)] Let $W_i$ be the exceptional locus of $\pi_i^{-}$.
  The set $\bigcup_i \mathrm{mld}(W_i;X_i,\delta D_i)$
  satisfies the ascending chain condition.
 \end{itemize}
 We could find the definition of $\mathrm{mld}(x;X,D)$ in 
 \cite[Definition 1.1]{Jet_scheme}.
 Since all $X_i$ are smooth, the condition (1) is satisfied by
\cite[Theorem 4.4]{Jet_scheme}. Moreover if we choose
 an integer $N$ for that $N\delta$ is an integer,
 all discrepancies are integers after $N$-multiple.
By \cite[Definition 1.1]{Jet_scheme}, minimal discrepancies are 
less than $\dim X+1$. Hence the condition (2) is
also satisfied. That is a contradiction.
\end{proof} 
 We complete the proof of Proposition  \ref{movable}.
\end{proof}
Let us start the proof of Theorem \ref{main}.
If $f^{*}D_{S}$ is nef, there are nothing to prove.
Assume that $f^{*}D_{S}$ is not nef.
By Proposition \ref{movable}, for the pair $(X,f^*D_S)$,
we have the following diagram:
\[
 \xymatrix{
  X \ar[rd]_{\pi^-_0} & &X_1 \ar[ld]^{\pi^+_0} 
  \ar@{<--}[ll]_{\phi_1} \ar[rd]_{\pi^-_1} && X_2 
  \ar[ld]^{\pi^+_1} 
  \ar@{<--}[ll]_{\phi_2} 
  \ar[rd]_{\pi^-_2} &&X_3 \ar[ld]^{\pi^+_2}
     \ar@{<--}[ll]_{\phi_3} &\cdots&  X_n 
     \ar[ld]^{\pi^+_{n-1}} 
       \\
  & Z_0  & & Z_1 & & 
  Z_2 & \cdots& Z_{n-1}  
 }
\]
where
$X_i$ are symplectic manifolds and
the proper transform 
$(\phi_{n}\circ \cdots \circ \phi_{1})_*(f^*D_S)$ is nef. 
We denote by $D_{i}$ the proper transform of $f^{*}D_{S}$ by
the composition map $\phi_{i}\circ \cdots \circ \phi_{1}$.
Then $D_{i}$ is $\pi_{i-1}^{+}$ ample.
%
%
\begin{lemma}\label{backword}
 The proper transform $D_{n}$ of $f^{*}D_S$ on $X_n$ is semiample
 and induces a Lagrangian fibration $f_n : X_n \to S_m$.
\end{lemma}
\begin{proof}
By Proposition \ref{movable}, $\phi_{n}\circ \cdots \circ \phi_{1}$ is isomorphic in a neighbourhood of a general 
fibre of $f: X \dasharrow S$. Hence
$X_n$ also admits an almost holomorphic Lagrangian fibration
$g : X_n \dasharrow S_{n}$.
Since
$\kappa (D_{n})=\kappa (D_S)$, $D_{n}$
is nef and good. By \cite[Theorem1.1]{MR2779478}, $D_{n}$
is semiample. 
Let
$f_n : X_n \to S_n$ be the induced morphism.
Since $D_{n}$ is trivial on a general fibre of $g$,
$f_n$ is a Lagrangian fibration.
\end{proof}
 We go back to the proof of Theorem \ref{main}.
 Let $E$ be
an irreducible component of the exceptional locus of 
$\pi_{n-1}$. 
We denote by
$F$ a general fibre of the induced
morphism $f_n|_E : E\to f_n (E)$. 
We note that $f_n (E) \ne S_n$ because the composition map
$\phi_{n}\circ \cdots \circ \phi_{1} : X \dasharrow X_n$
is isomorphic in a neighbourhood of a general fibre of 
the original almost holomorphic fibration $X\dasharrow S$
by Proposition \ref{movable}.
We take an embedded
resolution $\nu : Y \to X_n$ of $F$.
We denote by $\tilde{F}$ the proper transform of $F$ and
by $p$ the image of $F$. The following diagram shows the 
relationship of the defined objects: 
\[
\xymatrix{
    & Y   \ar[d]^{\nu} &  & \ar@{_{(}->}[ll] \tilde{F} \ar[d] \\
Z_{n-1} & X_n \ar[l]_{\pi_{n-1}^{+}} \ar[d]^{f_n}& 
           E \ar[d]^{f_n} \ar@{_{(}->}[l] 
           & F \ar@{_{(}->}[l] \ar[d]\\
    & S_n & \ar@{_{(}->}[l] f_n(E) & \ar@{_{(}->}[l] \{p\}
}
\]
We will prove the following two Claims which contradicts each other.
\begin{claim}\label{vanish}
 Let $\omega$ be a symplectic form on $X_k$
 whose restriction to
 a gereral fibre of $f_{n}$ is zero. Then $\nu^*\omega |_{\tilde{F}}$
 is zero.
\end{claim}

\begin{proof}
 We consider the following morphism:
\[
  R^2 (f_n\circ\nu)_*\mathcal{O}_{Y}\otimes k(p) \to 
  H^2 (Y_{p},\mathcal{O}) \to
  H^2 (\tilde{F},\mathcal{O})
\]
 where $k(p)$ stands for the residue field at $p$ and
 $Y_{p}$ the fibre of $f_n\circ \nu$ at $p$.
 By the assumption the restriction of $\bar{\omega}$
 to the restriction to the general fibre of $f_n\circ\nu$
 is zero. By \cite[Theorem 2.1]{MR825838},
 $R^2(f_n\circ\nu)_*\mathcal{O}_Y$ is torsion
 free. Hence the image of $\nu^*\bar{\omega}$ in
 $R^2(f_n\circ\nu)_* \mathcal{O}_Y\otimes k(p)$ 
 is zero. Therefore the image of $\nu^*\bar{\omega}$
 in $H^2 (\tilde{F},\mathcal{O})$ is also zero.
 This implies that $\nu^*\omega|_{\tilde{F}}=0$. 
\end{proof}

\begin{claim}\label{nonvanish}
 Under the same notation as Claim \ref{vanish},
 $\nu^*\omega|_F$ is not zero.
\end{claim}

\begin{proof}
We recall two geometric natures of the image of
 a birational contraction from a symplectic manifold.
Let $\pi_{n-1}^{+} (E)^{\circ}$ be the smooth locus 
of the image of $E$ and $\omega$ a symplectic form
on $X_n$. By \cite[Lemma 2.9]{MR2283801}, $\pi_{n-1}^{+} (E)^{\circ}$ carries
a symplectic form $\omega_E$ such that 
$\nu^*\omega_E=\omega$ on $E$.
Moreover 
\[
\dim \pi_{n-1}^{+} (E) = \dim X_{n} - 2(\dim X_{n} - \dim E) = 2\dim E - \dim X_{n}
\]
by \cite[Theorem 1.2 (ii)]{MR1966025}. 
We estimate the dimension of $\pi_{n-1}^{+} (F)$.
Since $f_n (E) \ne S_n$, $\dim F \ge \dim E - (1/2)\dim X_n + 1$.
Let $C$ be a contracted curve by $\pi_{n-1}^{+}$.
Since $D_{n}$ is $\pi_{n-1}^{+}$-ample, $D_{n}.C > 0$. 
Hence $C$
is not
contracted by $f_n$ 
and all curves contained in $F$ is not contracted by $\pi_{n-1}^{+}$.
 We have
\[
\dim \pi_{n-1}^{+} (F) \ge \dim E - \frac{1}{2}\dim X_{n} + 1.
\]
This implies $2\dim \pi_{n-1}^{+} (F) > \dim \pi_{n-1}^{+} (E)$ and we are done.
\end{proof}
 We complete the proof of Theorem \ref{main}. 
\end{proof}

\begin{proof}[Proof of Corollary \ref{semi-rigidity}]
 Let $Y$ be the common resolution of $X_1$ and $X_2$.
 We denote by 
 $\nu_i$, $(i=1,2)$ the birational morphisms
 $Y \to X_i$.
 \begin{claim}\label{bridge}
 Let $D_{S_1}$ be an ample divisor on $S_1$.
 There exists a $\mathbb{Q}$-Cartier divisor $D_{S_2}$
 such that 
 \[
 (f_1\circ \nu_1)^*D_{S_1}\sim (f_2\circ \nu_2)^*D_{S_2}
 \]
 \end{claim}
 \begin{proof}
 By Theorem \ref{main}, 
 $(\nu_2)_*(f_1\circ \nu_1)^*D_{S_1}$ is nef.
 By \cite[Theorem 1]{MR1783616}, $f_2$ is equidimensional.
 Since $X_2$ is smooth, $S_2$ is $\mathbb{Q}$-factorial.
 By the definition, $(\nu_2)_*(f_1\circ \nu_1)^*D_{S_1}$ does not
 contain a general fibre of $f_2$. Thus
  $f_2((\nu_2)_*(f_1\circ \nu_1)^*D_{S_1})$ defines a divisor of $S_2$.
  We define
  \[
  D_{S_2} := \mathrm{min}
  \{ \Delta |  \mbox{a $\mathbb{Q}$-divisor on
   $S_2$ such that $f^*_2\Delta -
  (\nu_2)_*(f_1\circ \nu_1)^*D_{S_1} \ge 0$}
  \}.
  \]
  If $f_2^*D_{S_2} \ne (\nu_2)_*(f_1\circ \nu_1)^*D_{S_1}$,
  then 
  $H:= - f_2^*D_{S_2} +(\nu_2)_*(f_1\circ \nu_1)^*D_{S_1} $
  is $f_2$-nef and $-H$ is effective.  
  That is a contradiction.
 \end{proof}
 We go back to the proof of Corollary \ref{semi-rigidity}.
 Let $C$ be a curve contained in a fibre of $f_2\circ \nu_2$. Then 
 $C$ is contracted by $f_1 \circ \nu_1$ because $(f_1\circ \nu_1)(C).D_{S_1}=0$
 by Claim \ref{bridge}. By the same argument, a curve contained in a fibre
 of $f_1 \circ \nu_1$ is contracted by $f_2 \circ \nu_2$. Thus we are done.
\end{proof}

\bibliographystyle{halpha}

\bibliography{almost_2011_Oct29}

\end{document}